\documentclass[12pt]{amsart}
\usepackage[utf8]{inputenc}
\usepackage[T1]{fontenc}
\usepackage[utf8]{inputenc}
\usepackage{lmodern}
\usepackage[a4paper,margin=1.9cm]{geometry}
\usepackage{amsmath,amssymb,amsfonts,amsthm}
\usepackage{mathtools}
\usepackage{bm}
\usepackage{enumitem}
\usepackage{hyperref}
\usepackage{ upgreek }
\usepackage{amssymb}
\usepackage{amsthm}
\usepackage{amsmath}
\usepackage{amssymb}
\usepackage{amscd}
\usepackage{graphicx}
\usepackage{fancyhdr}
\usepackage{hyperref}
\hypersetup{
    colorlinks=true,
    linkcolor=black,
    filecolor=magenta,      
    urlcolor=blue,
    citecolor=black,
    pdftitle={Two-Dimensional Billiards Are Turing Complete},
    pdfpagemode=FullScreen,
    }

\graphicspath{ {/imagen/} }

\newcommand{\Z}{\mathbb{Z}}

\newtheorem{lemma}{Lemma}
\newtheorem{proposition}{Proposition}
\newtheorem{theorem}{Theorem}
\newtheorem{corollary}{Corollary}
\theoremstyle{definition}
\newtheorem{remark}{Remark}
\newtheorem{definition}{Definition}

\title{Two-Dimensional Billiards Are Turing Complete}

\author{Eva Miranda}
\address{ Eva Miranda,
Laboratory of Geometry and Dynamical Systems and SYMCREA research unit, Department of Mathematics,  Universitat Polit\`{e}cnica de Catalunya, Barcelona, Spain  \& Centre de Recerca Matemàtica, CRM \& Forschungsinstitut für Mathematik, ETH Zürich, Zürich, Switzerland}
\email{eva.miranda@upc.edu}

\author{Isaac Ramos}
\address{ Isaac Ramos, Department of Mathematics, ETH Zürich, Zürich, Switzerland \& SYMCREA research unit at Universitat Polit\`{e}cnica de Catalunya, Barcelona, Spain}
\email{iramos@ethz.ch}
\thanks{Both authors are partially supported by the project “Computational, dynamical and geometrical complexity in fluid dynamics” (COMPLEXFLUIDS), Ayudas Fundación BBVA a Proyectos de Investigación Científica 2021. Eva Miranda is supported by the Catalan Institution for Research and Advanced Studies via an ICREA Academia Prize 2021 and by the Alexander Von Humboldt Foundation via a Friedrich Wilhelm Bessel Research Award. Miranda is also supported by the Spanish State
Research Agency, through the Severo Ochoa and Mar\'{\i}a de Maeztu Program for Centers and Units
of Excellence in R\&D (project CEX2020-001084-M) and both authors are supported by the project PID2023-146936NB-I00   and the joint AEI-DFG project PCI2024-155042-2 entitled \emph{Celestial Mechanics, Hydrodynamics, and Turing Machines} and by the Spanish State Research Agency MCIU/AEI / 10.13039/501100011033/ FEDER, UE.  Miranda thanks the FIM at ETHZ Zurich for their hospitality during her stay in Zürich as Nachdiplom lecturer in the Fall Semester of 2025 when this article was written. Isaac Ramos is supported by a fellowship from
``La Caixa'' Foundation (ID 100010434) with code LCF/BQ/PFA25/11000070. }
\begin{document}
\maketitle

\begin{abstract}
We show that two-dimensional billiard systems can simulate universal Turing machines. Billiards serve as idealized models of particle motion with elastic reflections and arise naturally as limits of smooth Hamiltonian systems under steep confining potentials, and as an exact reformulation of hard-sphere gas dynamics. By invoking the undecidability of the halting problem, originally established by Turing, our results show that undecidable trajectories arise in these physically natural billiard-type models. We further discuss, at the level of analogy, a connection to collision-chain limits in celestial mechanics, where near-collision dynamics exhibit billiard-like symbolic itineraries.
\end{abstract}
\subsection*{Significance statement}

Billiards are a textbook model of deterministic motion: a particle moves freely and reflects specularly from rigid walls. We show that, even in two dimensions, billiard trajectories can simulate arbitrary Turing machines. This universality implies a sharp limit on prediction: there is no general algorithm that can decide basic questions such as whether a trajectory is periodic. Because billiards also arise as limits of smooth Hamiltonian systems with increasingly steep confining potentials, these algorithmic barriers are not confined to idealized hard-wall models. Our results place undecidability, alongside chaos, as a fundamental obstruction to long-term prediction even in low-dimensional classical dynamics.

\section{Introduction}

Billiards are among the simplest dynamical systems: a point particle moves freely inside a domain and undergoes specular reflection at the boundary. Since Birkhoff’s formulation of the billiard map, they have served as canonical models for Hamiltonian and symplectic dynamics with singular interactions \cite{Birkhoff1917,Birkhoff1927}. Despite their elementary definition, billiards exhibit an extraordinary range of behaviors, from complete integrability to strong chaos. Sinai’s dispersing billiards provided foundational examples of hyperbolic and ergodic dynamics \cite{Sinai1963}, while polygonal billiards connect to flat geometry, translation surfaces, and Teichm\"uller dynamics \cite{Mirzakhani2008,KerckhoffMasurSmillie1986}.

Billiards also play a central role in physics. In kinetic theory, systems of hard spheres are equivalent to billiard flows in configuration space \cite{ChernovMarkarian2006}. In quantum mechanics, quantum billiards model systems such as quantum dots and microwave cavities, where a particle is confined by rigid boundaries, and in the semiclassical regime, many features of the quantum dynamics are determined by the geometry and trajectories of the corresponding classical billiard \cite{Berry,Gutzwiller1990,Haake2010}. Importantly, billiards can be realized as limits of smooth Hamiltonian systems by replacing hard walls with steep confining potentials that diverge near the boundary. In this soft-wall limit, smooth Hamiltonian flows converge (excluding trajectories that hit the boundary nearly tangentially) to billiard motion, allowing billiard results to be interpreted within standard classical mechanics.

Given the richness of billiard dynamics, a natural question arises: can billiards realize the full complexity of computation? Since the pioneering work of Moore \cite{Moore1,Moore2}, it has become clear that many continuous dynamical systems can simulate algorithms and thereby exhibit undecidability phenomena rooted in Turing’s Halting problem \cite{turing} (for a historical discussion on the attribution see also \cite{joel}). At a more foundational level, the question of whether physical systems can genuinely realize such Turing-universal behavior has been central since the very origins of computation, notably in the seminal works of Feynman \cite{Feynman1982Simulating} and Wolfram \cite{WolframUndecidability}.
Universal computation has since been embedded in a wide range of settings, including polynomial differential equations, mechanical systems with tailored potentials, cellular automata and hydrodynamic flows \cite{KoiranCosnardGarzon,GracaCampagnoloBuescu_Computability,BournezGracaHainry,Wolfram_CA,Wolfram_Fluids,tao1,CardonaMirandaPeralta_conjectureMoore, CardonaMirandaPeralta_survey, CardonaMirandaPeraltaPresas, CardonaMirandaPeralta_universality, CardonaMirandaPeralta_euclidean, GonzalezMirandaPeralta_universality, BrueraCardonaMirandaPeraltaSalasSalo2024}. A unifying framework behind several of these results is Topological Kleene Field Theory, which encodes computable functions into flows on manifolds with boundary so that dynamical questions translate into questions of computation \cite{GonzalezMirandaPeralta_TKFT}. 

A Turing machine is a simple mathematical model of computation. It consists of a finite set of internal states, a tape whose entries belong to a fixed alphabet, and a transition rule that prescribes, at each step, how the machine updates its state, modifies the tape, and moves along it. Despite its elementary definition, there exist universal Turing machines that can simulate any algorithmic procedure.

When a dynamical system can reproduce, under a suitable encoding, the evolution of a universal Turing machine, the system is said to be \emph{Turing complete}. This provides a precise bridge between computation and dynamics. In particular, natural questions about the trajectories of a Turing complete system may inherit the undecidability of the halting problem: there is then no algorithm that can decide the corresponding dynamical question for all possible inputs in finite time.

Within this program, billiards present a stringent test case. Moore suggested that billiard systems in sufficiently high dimension, most notably three-dimensional billiards, could support universal computation, while also arguing that genuinely low-dimensional systems might lie below the universality threshold \cite{Moore1,Moore2}. Subsequent work achieved Turing completeness for billiard-like models by adding dynamical ingredients such as many-ball systems \cite{FredkinToffoli}, three-dimensional walls \cite{ReifTygar1994}, and moving walls or internal memory encoded in the particle’s velocity, for example, in computational pinball machines \cite{pinball}. \emph{Here we show that none of these additional mechanisms is necessary: classical planar billiards with one particle and fixed walls already realize universal computation.} 

We take inspiration from the construction of Topological Kleene Field Theory and adapt it to billiard flows. More precisely, for each Turing machine, we construct a two-dimensional billiard table whose trajectories simulate its computation. As a consequence, natural decision problems for planar billiards, including reachability and periodicity questions, are algorithmically undecidable. Our approach is based on an encoding that is compatible with billiard geometry. We encode computation states as points of a one-dimensional interval and let the tape head move rather than shifting the entire tape. In our realization, the shift arises naturally from reflections on parabolic walls, while the read–write operation is more involved but remains fully compatible with a physical billiard system. A key geometric point is that the billiard tables arising in our construction are not polygonal and should not be thought of as having finitely many linear walls. The read--write operations are implemented by special control walls with an oscillatory profile. These walls are organized into a countable family indexed by the head position $k\in\mathbb Z$, but the resulting billiard table is nevertheless \emph{finitely piecewise smooth}\footnote{Throughout this paper, ``smooth'' is understood to mean $C^\infty$.}: away from finitely many singular boundary points, each control wall is obtained by smoothly joining together countably many explicit linear segments into a single smooth arc. Thus the construction requires infinitely many local oscillations, but only finitely many singularities in the boundary of the billiard table.

Beyond its intrinsic mathematical interest, Turing completeness in planar billiards has broader implications for classical mechanics. Since billiards arise as hard wall limits of smooth Hamiltonian systems, undecidability in billiards points to intrinsic algorithmic limits in steep potential Hamiltonian dynamics. Billiard-type descriptions also appear in several physical regimes, including hard sphere gases and collision-dominated limits in celestial mechanics. In particular, models arising near close encounters in celestial mechanics, including variants of the three-body problem, often admit effective billiard descriptions. Our results therefore point to the possibility that collision dominated regimes may exhibit qualitative behaviors that are not algorithmically decidable, even when the governing equations are deterministic. This suggests that undecidability may enter classical celestial mechanics alongside chaos, not as another form of sensitive dependence on initial conditions, but as a distinct obstruction to prediction. Such algorithmic barriers could affect the mathematical formulation of long-term questions concerning stability, escape, and ejection in planetary systems.

\noindent\textbf{Outline.}
Section 2 reviews billiards, Turing machines and simulation of Turing machines by dynamical systems. Section 3 presents the construction of billiards associated with a Turing machine and proves Turing completeness. Section 4 discusses consequences and physical interpretations.
Section 5 presents some conclusions.

\textbf{Acknowledgements:} We thank FIM and ETH Zürich for providing an ideal environment for this collaboration. We are also grateful to Stephen Wolfram for his visit in October 2025, which inspired several related questions that ultimately led to this work. We further thank Ángel González-Prieto for insightful discussions that shaped our encoding of the computational states of a Turing machine into a one-dimensional interval. We also thank John Baez, Robert Cardona, Urs Frauenfelder, Joel David Hamkins and Daniel Peralta-Salas  for their interest. Last but not least, we are grateful to the referees for their careful and thoughtful reading of the manuscript. Their constructive comments, including those that challenged the scope and presentation of our initial submission, helped us refine the structure of the article and arrive at a significantly stronger version. In particular, one of the pathological examples presented in Remark \ref{rmk: example referee} was suggested by one of the referees.

\section{Preliminaries}

\subsection{Billiards}
\label{sec:billiards-Rn}
Planar billiards describe the motion of a point particle undergoing free flight
\[
\ddot q(t)=0
\]
inside a domain \(B \subset \mathbb{R}^2\), with elastic reflections at the boundary:
\[
v^+ = v^- - 2\langle v^-, n\rangle n,
\]
where \(n\) denotes the inward unit normal to \(\partial B\) at the point of impact, and
\(v^-\) and \(v^+\) are the incoming and outgoing velocities, respectively. Thus the speed
is preserved and the angle of incidence equals the angle of reflection.

The billiard dynamics can be formulated as a dynamical system. Its natural phase space is
\[
\mathcal P_B=\{(q,v)\in B\times \mathbb R^2:\ |v|=1\},
\]
together with the boundary identifications prescribed by the elastic reflection law. Equivalently,
one may view the phase space as \(B\times S^1\), away from the boundary identifications. The
billiard flow
\[
\varphi_t:\mathcal P_B\longrightarrow \mathcal P_B
\]
is obtained by following the free motion \(q(t)=q_0+t v_0\) until the first collision with
\(\partial B\), applying the reflection rule, and iterating this procedure. At singular points
of the boundary, such as corners or points where the normal is not defined, the flow may fail
to be defined. Accordingly, \(\varphi_t\) is naturally considered on the subset of phase space
consisting of trajectories that avoid such singularities.

Despite their elementary definition, billiards arise ubiquitously in physics as exact or
effective models for systems with hard constraints, singular interactions, or piecewise-smooth
dynamics. Depending on the geometry and regularity of the boundary \(\partial B\), one
distinguishes several important classes of billiards, such as polygonal billiards with piecewise
linear boundary and smooth billiards, including elliptic tables. More generally, one often
assumes that \(\partial B\) is \emph{finitely piecewise smooth}, namely smooth except for
finitely many singular points. In this article, we work within this general setting and do not
restrict ourselves to polygonal or convex billiards.

\subsection{Turing machines}
This celebrated computability model captures the essence of a computer from a mathematical point of view. Due to its intuitive formulation, it has become one of the most standard and widely adopted models of computation. Formally, a (binary) Turing machine consists of a tuple
\[
M = (Q, q_0, Q_{\mathrm{halt}}, \delta),
\]
where $Q$ is a finite set, whose elements are the \emph{states} of $M$, $q_0 \in Q$ is the \emph{initial state}, $Q_{\mathrm{halt}} \subsetneq Q$ is the subset of \emph{halting states}, and
\[
\delta : Q \times \mathcal{A} \to Q \times \mathcal{A} \times \{-1,1\}
\]
is the \emph{transition function}. Here $\mathcal{A} = \{0,1\}$ is the tape alphabet.

From this information, we can define a dynamical system. Let $\Lambda \subset \mathcal{A}^{\mathbb{Z}}$ be the set of finitely supported two-sided sequences 
$t : \mathbb{Z} \to \mathcal{A} = \{0,1\}$. We refer to each $t \in \Lambda$ as a \emph{tape state}. The set of computation states $\mathcal{S}$ consists of words of the form
\[ \mathcal{S}=\{ (t,q,k) \in \Lambda \times Q\times \Z \}\]
where $t \in \Lambda$ is a tape state, $q \in Q$ is the current internal state of the machine and $k\in \Z$ is the position of the head\footnote{Notice that in previous works \cite{Moore1,Moore2,CardonaMirandaPeraltaPresas,CardonaMirandaPeralta_euclidean,GonzalezMirandaPeralta_TKFT,DyhrGonzalezMirandaPeralta_Navier}, the head position is not treated as part of the Turing machine's state.  These definitions fix the head at position \(0\) and simulate head moves by shifting the entire tape.  We adopt the more intuitive approach of including the head position explicitly, which simplifies our encoding into a billiard system.}. At certain points in our construction, when $q$ is not relevant, we will abuse terminology and refer to $(t,k)$ alone as the computation state. We then define
\[
\Delta_M : \mathcal{S} \to \mathcal{S}
\]
by the following rule. Given a computation state $(t,q,k)$, the Turing machine evaluates $\delta(q, t_k) = (q', s, \varepsilon)$ and sets $\Delta_M(t,q,k) = (t', q',k+\varepsilon)$, where the new tape $t'$ satisfies
\[
t'_n = t_{n} \quad \text{for } n \neq k, 
\qquad 
t'_{k} = s.
\]

The map $\Delta_M$ encodes the intuitive idea that $t$ represents a two-sided tape, and $M$ has a read--write head currently positioned at cell $k$. At each iteration, the head reads the symbol underneath it, and based on the present internal state $q$ of the machine, it writes a new symbol $s$ on the tape, moves either left ($\varepsilon = -1$) or right ($\varepsilon = 1$), and updates the internal state to $q'$. In this way, for a given initial tape $t$, the map $\Delta_M$ generates a dynamical system on $\mathcal{S}$ starting from $(t, q_0,0) \in \mathcal{S}$. If, at some point, the system reaches a configuration $(t', q',k')$ with $q' \in Q_{\mathrm{halt}}$, the computation is said to halt.

In the following section, we will see that the reversibility condition is essential for building well-structured dynamical systems capable of simulating universal computation. The reversibility of a Turing machine can be defined in several equivalent ways. For our purposes, the formulation based on the global transition function $\delta$ is the most convenient.

\begin{definition}
A Turing machine $M = (Q, q_0, Q_{\mathrm{halt}}, \delta)$ is \emph{reversible} if the global transition function $\Delta_M$ is injective.
\end{definition}

In the seminal paper \cite{Bennet73}, Bennett demonstrated that any computation can be carried out by a reversible Turing machine, as stated in the following theorem.
\begin{theorem}\label{thm: reversibility}
    \cite{Bennet73} For every Turing machine, there exists an equivalent reversible Turing machine.
\end{theorem}
Hence, focusing exclusively on reversible Turing machines does not impose any real restriction.

There is a geometric perspective for interpreting Turing machines as graphs, commonly called \emph{finite state machines}. Consider a Turing machine
\[
M = (Q, q_0, Q_{\mathrm{halt}}, \delta).
\]
We associate to $M$ a finite state machine $\mathcal{G}_M$, defined as the following labelled directed graph:

\begin{itemize}
    \item The vertices of $\mathcal{G}_M$ are the states $Q$ of $M$.

    \item Each non-halting state $q \notin Q_{\mathrm{halt}}$ has two outgoing edges connecting to the new states determined by the transition function, 

    \item The halting states in $Q_{\mathrm{halt}}$ are represented as vertices with no outgoing edges.
\end{itemize}
The finite state machine $\mathcal{G}_M$ corresponding to the Turing machine $M$ thus provides a graphical representation of the evolution of $M$’s internal states during a computation. The computation begins at the initial vertex $q_0$ and, at each step, the system moves along the edge determined by the symbol currently read on the tape, updating the tape and shifting the head left or right according to the label of that edge. As we will explain in the proof of the main theorem, we will use this geometrical representation to encode any Turing machine as a billiard flow.

\begin{remark}\label{rem:normalizations}
For the remainder of the paper, we may assume, without loss of generality, that the Turing machines under consideration satisfy the following properties (see also, \cite{GonzalezMirandaPeralta_TKFT}):
\begin{enumerate}
    \item they are reversible;
    \item they have a unique halting state $q_f$;
    \item no transition enters the initial state $q_0$.
\end{enumerate}
The first assumption follows from Theorem~\ref{thm: reversibility}. The other two are convenient normalizations of the finite-state machine. If a reversible Turing machine has several halting states, one may attach a finite auxiliary gadget that funnels all of them into a single distinguished halting state. Similarly, if there are transitions entering the initial state, one may introduce a new state $q_0'$ that receives all incoming edges of $q_0$, and connect $q_0$ to $q'_0$ by an auxiliary edge. These modifications alter only the presentation of the machine and not the partial function that it computes.
\end{remark}

\subsection{Simulation of Turing machines by dynamical systems}

We now make precise what it means for a dynamical system to simulate a
Turing machine, and in particular what it means for a dynamical system to be
Turing complete. The guiding point is simple: a Turing machine already defines
a discrete dynamical system. The question is then which other dynamical
systems, either discrete or continuous, can reproduce this computation inside
their own phase space.

The definitions below are inspired by the classical framework of symbolic
dynamics and generalized shifts, where computational states are represented by
points of a phase space and the evolution of the system reproduces the
evolution of the computation
\cite{Moore1,Moore2,tao1,CardonaMirandaPeraltaPresas,
CardonaMirandaPeralta_euclidean,BrueraCardonaMirandaPeraltaSalasSalo2024,
GonzalezMirandaPeralta_TKFT,DyhrGonzalezMirandaPeralta_Navier}.

We begin with the elementary notion of a dynamical system. A
\emph{discrete dynamical system} on a topological space $X$ is a continuous
map
\[
F:X\to X .
\]
A \emph{continuous dynamical system} on $X$ is a family of maps $\varphi_t:X\to X,\qquad t\in\mathbb R,$ depending continuously on $t$, such that
\[
\varphi_0=\operatorname{id}_X,
\qquad
\varphi_{t+s}=\varphi_t\circ\varphi_s
\]
for all $s,t\in\mathbb R$.

Let $M=(Q,q_0,Q_{\mathrm{halt}},\delta)$ be a Turing machine. Its space of configurations is $\mathcal{S}=\Lambda\times Q\times \mathbb Z ,$ where $\Lambda\subset \{0,1\}^{\mathbb Z}$ is the set of finitely supported
tapes, $Q$ is the finite set of internal states, and the last factor records
the position of the head. Thus a point $(a,q,k)\in \mathcal{S}$ records the complete
instantaneous state of the computation: the tape $a$, the internal state
$q$, and the head position $k$. We denote by
$\Delta_M:\mathcal{S}\to \mathcal{S}$ the global transition map of the machine.

We regard $\mathcal{S}$ as a topological space in the natural way. The sets $Q$
and $\mathbb Z$ are endowed with the discrete topology, while $\Lambda$ is
endowed with the metric
\[
d_\Lambda(a,a')=\sum_{n\in\mathbb Z}
\frac{|a_n-a'_n|}{2^{|n|+1}} .
\]
Equivalently, $\mathcal{S}$ may be endowed with the metric
$d\bigl((a,q,k),(a',q',k')\bigr)
=
d_\Lambda(a,a')
+
\delta_Q(q,q')
+
\delta_{\mathbb Z}(k,k') ,
$
where $\delta_Q$ and $\delta_{\mathbb Z}$ are the discrete metrics. With
this topology, the global transition map
$\Delta_M:\mathcal{S}\to \mathcal{S}$
is continuous. Hence every Turing machine naturally determines a discrete
dynamical system $\Delta_M$ on $\mathcal{S}$.

We can now ask when another dynamical system reproduces this intrinsic
configuration dynamics. The idea is that each computational state is encoded as a point of a phase space, and that the evolution of the ambient dynamical system advances these encoded points exactly as the Turing machine advances
its configurations. 

A crucial aspect of any notion of simulation is that the encoding of
computational states into the phase space of the dynamical system should itself
be effective. Otherwise, the computational content of the Turing machine could
be hidden in the encoding rather than being generated by the
dynamics. For this reason, we require the encoding map to be \emph{constructible}. First, recall that a point $x\in X$ of a smooth manifold embedded in Euclidean space\footnote{We restrict attention to smooth manifolds embedded in Euclidean space, which is sufficient for the billiard systems considered here and provides a direct notion of what it means for a Turing machine to output a real number. More intrinsic formulations are possible: in particular, computable points can be defined in the general framework of \emph{computable metric spaces}. See \cite{rojas}.} is said to be
\emph{computable} if there exists an algorithm that, given $n\in\mathbb{N}$,
produces an approximation of $x$ with error at most $2^{-n}$. Here, we go beyond this notion and ask the encoding to be constructible, in the sense that there exists an algorithm that can output the exact coordinates for each point of the encoding in a finite amount of time. 

We can now ask when another dynamical system reproduces this intrinsic
configuration dynamics. The idea is that each computational state is encoded as
a point of a phase space, and that the evolution of the ambient dynamical
system advances these encoded points exactly as the Turing machine advances its
configurations.

\begin{definition}\label{def: discrete system}
Let $X$ be a smooth manifold embedded in Euclidean space and $F$ be a discrete dynamical system on it. We say that $F$ \emph{simulates} the Turing machine $M$ if there exists a continuous map $\Phi:S\hookrightarrow X$ such that
\[
F\circ\Phi=\Phi\circ\Delta_M.
\]
Moreover, we require $\Phi$ to admit a constructible description, in the sense that there exists a Turing machine that outputs the exact coordinates of $\Phi(s)$ for each $s\in S$ in a finite amount of time.
\end{definition}

In particular,

\[
F^n(\Phi(s))=\Phi(\Delta_M^n(s))
\]

for every computational state $s\in S$ and every $n\geq0$. Thus, the restriction of the dynamical system to the invariant subset $\Phi(S)$ replicates the dynamics of the Turing machine.

For continuous-time systems, it is often too restrictive to require that a fixed amount of time corresponds to one computational step. Indeed, in many physical systems, such as billiard flows we describe later, different transitions may require different amounts of time. We therefore adopt the following more flexible notion.

\begin{definition}\label{def: cont simulation}
Let \(X\) be a smooth manifold embedded in Euclidean space, and let
\(\varphi_t\) be a continuous dynamical system on \(X\). We say that
\(\varphi_t\) simulates the Turing machine \(M\) if there exist an
injective continuous map $\Phi\colon S \hookrightarrow X$ and a computable function $\tau\colon S \longrightarrow (0,\infty)$
such that
\[
\varphi_{\tau(s)}(\Phi(s))=\Phi(\Delta_M(s))
\]
for every computational state \(s\in S\). As before, we require that
\(\Phi\) admit a constructible description, in the sense of Definition~2.

Iterating this relation, if \(t_0=0\) and
$t_{n+1}=t_n+\tau(\Delta_M^n(s)),$
then $\varphi_{t_n}(\Phi(s))=\Phi(\Delta_M^n(s))$ for all \(n\geq 0\). Moreover, we require that $\varphi_t(\Phi(s))\notin \Phi(S)$, for all \(t\notin\{t_0,t_1,t_2,\ldots\}\).
\end{definition}

We finally recall the notion of universality. A Turing machine $U$ is called
\emph{universal} if, for every Turing machine $M$ and every input $w$, there
is an effective encoding $\langle M,w\rangle\in S_U$
such that the computation of $U$ starting from $\langle M,w\rangle$
reproduces the computation of $M$ on input $w$. In particular, 
$U \text{ halts on } \langle M,w\rangle$
if, and only if, $M \text{ halts on } w.$

\begin{definition}\label{def: Turing completeness}
A discrete dynamical system is called \emph{Turing complete} if it simulates a
universal Turing machine.

A continuous dynamical system is called \emph{Turing complete} if it simulates
a universal Turing machine in the variable time sense defined above.
\end{definition}

An important consequence of these definitions is that algorithmic properties of universal Turing machines are transferred to the dynamical system. Let
\[
H:=\Phi(S_H)
\]
denote the image of the halting states of a universal Turing machine $M$. Since $\Phi$ is constructible, every input
$t$ can be effectively translated into the point
$\Phi(s_t)\in X$, where $s_t$ is the initial computational state associated
with the input $t$.

Suppose now that $(X,F)$ or $(X,\varphi_t)$ simulates the universal Turing
machine $M$. Then
\[
\Delta_M^n(s_t)\in S_H
\]
for some $n\geq 0$
if and only if
\[
F^n(\Phi(s_t))\in H
\]
in the discrete-time case, or equivalently
\[
\varphi_{t_n}(\Phi(s_t))\in H
\]
for some $n\geq0$ in the continuous-time case. Therefore, deciding whether the
orbit of a computable point ever reaches the halting set is equivalent to
deciding whether the corresponding Turing machine halts on the given input.

Since the halting problem is undecidable, it follows that no algorithm can
decide, in general, whether a trajectory starting from a computable initial
condition eventually reaches the halting set. Thus the simulation of a
universal Turing machine immediately yields dynamical systems exhibiting
undecidable orbit properties.

\begin{remark}\label{rmk: example referee}
    In the discussion of undecidable dynamics, it is essential that our notion of simulation reproduces the computation step by step, rather than merely encoding its final input--output behavior. Otherwise, undecidability could be hidden in the construction of the dynamical system itself instead of emerging from its dynamics.

To illustrate this point, consider a Turing machine $M$. Suppose we are allowed to construct a billiard table after determining whether $M$ halts on any given input $t$. If $M$ halts on $t$, we arrange the walls so that the trajectory starting from a computable initial condition associated with $t$ is reflected into a designated output region. If $M$ does not halt on $t$, we instead arrange the walls so that the same initial condition enters another region of the billiard table and stays there for the rest of its trajectory. In this way, deciding whether the orbit reaches the designated output region becomes equivalent to the halting problem. However, such a construction does not provide a meaningful dynamical realization of computation. The computation is never carried out by the dynamical system itself; rather, the answer has already been encoded into the geometry of the billiard table. Thus, the source of undecidability lies in the description of the system, not in its evolution.

A similar issue arises with the encoding map $\Phi$ if one merely requires it to be computable rather than \emph{constructible}. Indeed, if $\Phi$ is only computable in the sense of computable analysis, the computation can still be hidden in the encoding instead of the dynamics. For example, one could encode a computational state $s$ by its entire future orbit
\[
\Phi(s)=(s,\Delta_M(s),\Delta_M^2(s),\ldots),
\]
and let the dynamical system be the left shift on such sequences. Then the relation
\[
F\circ\Phi=\Phi\circ\Delta_M
\]
still holds, so the system formally simulates the Turing machine. However, the computational complexity is no longer generated by the dynamics, since the entire future evolution of the machine has already been encoded in the initial condition. This construction satisfies the usual notion of computability because the space of sequences can be endowed with a metric in which the contribution of later coordinates decays rapidly, so that arbitrarily accurate approximations of $\Phi(s)$ require only finitely many iterates of $\Delta_M$. For this reason, we impose the stronger effectiveness condition that, for every finitely supported computational state $s$, the encoded point $\Phi(s)$ admits an exact finite description that can be produced by a Turing machine in finite time. This prevents undecidability from being hidden in the encoding itself, as observed already in \cite{CardonaMirandaPeraltaPresas}.

Consequently, a genuine simulation requires both a constructible encoding of every computational state and a dynamical evolution that reproduces each transition of the Turing machine. In this setting, undecidability genuinely arises from the long-term behavior of the dynamics and cannot be attributed to information hidden in the construction.
\end{remark}

Examples of Turing complete dynamical systems include piecewise linear maps of the square, polynomial differential equations, mechanical systems with tailored potentials, cellular automata and hydrodynamic flows \cite{Moore1,Moore2,GracaCampagnoloBuescu_Computability,BournezGracaHainry,Wolfram_CA,Wolfram_Fluids,tao1,CardonaMirandaPeralta_conjectureMoore, CardonaMirandaPeralta_survey, CardonaMirandaPeraltaPresas, CardonaMirandaPeralta_universality, CardonaMirandaPeralta_euclidean, GonzalezMirandaPeralta_universality, BrueraCardonaMirandaPeraltaSalasSalo2024}. In the following section we prove that planar classical billiards are also capable of simulating any Turing machine, and hence they are Turing complete in the sense of Definition \ref{def: Turing completeness}.

\section{Turing completeness of two-dimensional billiards}
In this section we state our main theorem on two-dimensional billiards, namely that for each (reversible) Turing machine, there exists a two-dimensional billiard system that simulates it in the sense of Definition \ref{def: cont simulation}. Then, we sketch our proof, adapted from the Topological Kleene Field Theory construction \cite{GonzalezMirandaPeralta_TKFT}, and later develop the tools needed for the full argument.

The statement of our main theorem is the following:
\begin{theorem}\label{thm: main result}

Let \(M=(Q,q_0,Q_{\mathrm{halt}},\delta)\) be a reversible Turing machine. Then there exists an explicitly constructible bounded
planar billiard table with finitely piecewise smooth boundary \(B\) whose billiard flow simulates $M$ in the sense of Definition \ref{def: cont simulation}. In particular, there are embeddings $\Phi_q:I\hookrightarrow B,$ one for each
state $q\in Q$, such that:
\begin{enumerate}

\item The image \(\Phi_{q_0}(I)\) is contained in \(\partial B\). For every computation state \((t,q_0,k)\), the point $\Phi_{q_0}(x_{t,k})$
is a rational point of the boundary and determines a billiard trajectory by launching the particle orthogonally to \(\partial B\).

\item If
\[
(t_n,q_n,k_n)=\Delta_M^n(t_0,q_0,k_0),
\]
then the billiard trajectory starting orthogonally from
$\Phi_{q_0}(x_{t_0,k_0})$ passes successively through the points $\Phi_{q_n}(x_{t_n,k_n})$ for $n=0,1,2,\ldots.$

\item The correspondence
\[
(t,q,k)\longleftrightarrow \Phi_q(x_{t,k})
\]
is finitely computable and bijective onto its image. In particular, the tape contents, the internal state, and the head position can be algorithmically recovered from the position of the billiard particle whenever it passes through one of the sections \(\Phi_q(I)\).
\item The image \(\Phi_{q_f}(I)\) of the halting state is contained in \(\partial B\), and the computation halts if and only if the corresponding billiard trajectory reaches \(\Phi_{q_f}(I)\).
\end{enumerate}
\end{theorem}

To prove this theorem, we will adapt the construction of computational bordisms within the framework of Topological Kleene Field Theory \cite{GonzalezMirandaPeralta_TKFT} to the setting of billiards, associating to each (reversible) Turing machine a billiard table whose relevant billiard trajectories mirror the machine's computations.

In this proof, we deeply exploit an alternative representation of a Turing machine as a finite state machine. As we previously discussed, given a Turing machine $M$, we can express
it as a directed finite graph where its nodes correspond to the states of $M$
and the edges represent the simple operations performed by the machine at each
step. We then transform the graph representing \(M\) into a billiard table. Using Cantor-based encodings, we can algorithmically associate to each tape state $t\in \Lambda$ and head position $k\in \Z$, a rational point $x_{t,k}\in I=[0,1]$. For each state \(q_i\), we fix a segment \(\Phi_i(I)\subset B\), where \(\Phi_i:I\hookrightarrow B\) is a rational translation of the unit interval into the billiard table, chosen so that these segments are transverse to the relevant billiard trajectories. In this sense, the billiard trajectory passes through the point \(\Phi_i(x_{t,k})\) precisely when the current computation state is \((t, q_i, k)\). In particular, $\Phi_0(I)$, which corresponds to the starting states of the Turing machine, will be contained on $\partial B$. Each edge of the graph is represented by a ``corridor'' in the billiard table connecting two such segments.  Along these segments, the billiard dynamics reproduce the operations encoded by that edge, and accordingly transform the relevant beam of billiard trajectories.

See Figure~\ref{fig: graph to billiard} for a sketch of the billiard associated with a simple Turing machine.

\begin{figure}[h]
    \centering
    \includegraphics[width=0.9\linewidth]{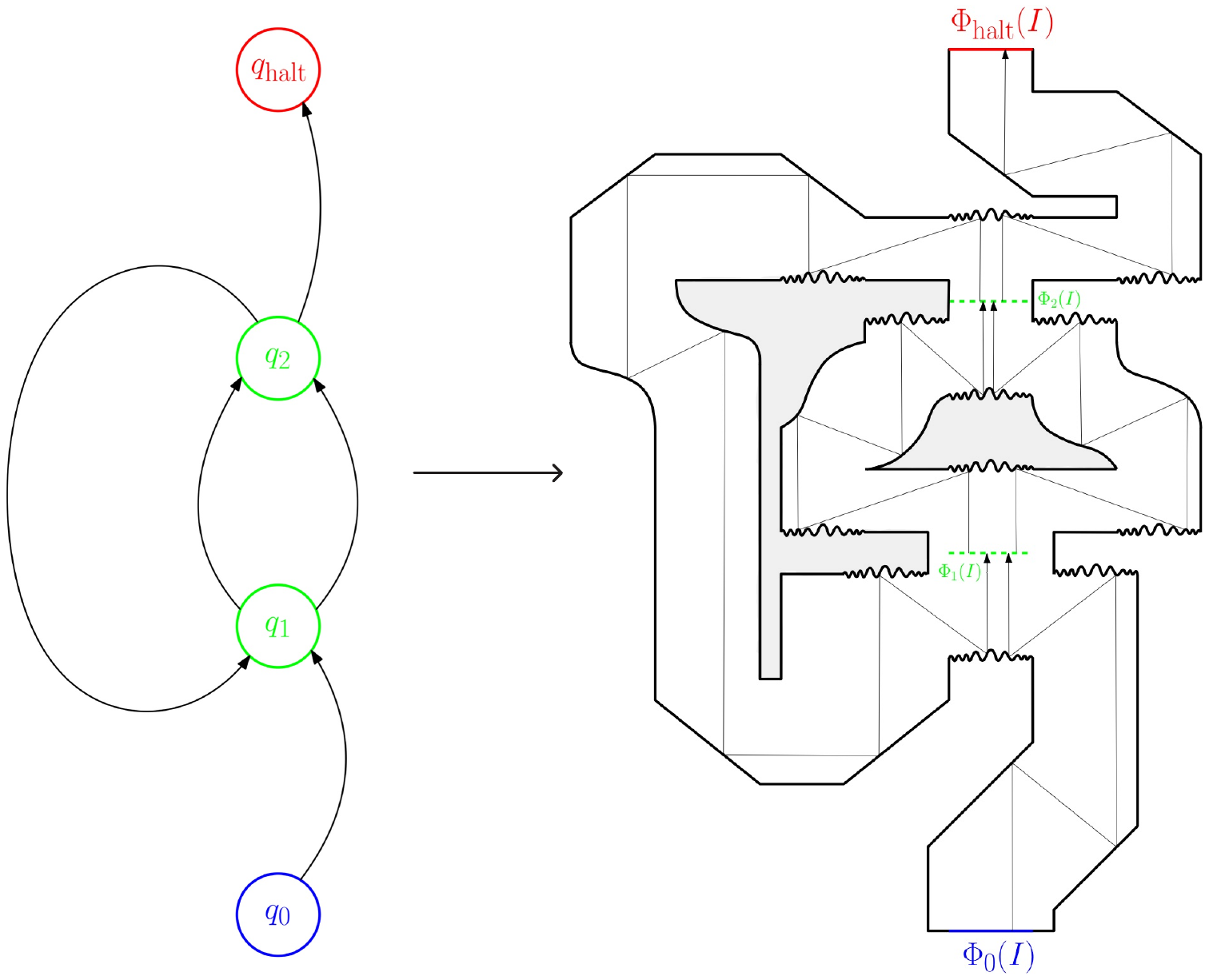}
    \caption{Transforming the graph representation of a Turing machine into a billiard.}
    \label{fig: graph to billiard}
\end{figure}
\begin{remark}
    Notice that using this proof method, the billiard table constructed has the graph of the associated Turing machine as a strong deformation retract. In particular, our construction requires a billiard table with a non-trivial topology, where we place an obstacle in the billiard for each of the generators of the fundamental group of the graph. In this context, one may ask the following question: is it possible to construct a billiard associated Turing machine $T$ whose number of obstacles is strictly less than the first Betti number of the graph corresponding to $T$? This is reminiscent of the connection between computational complexity and topological complexity. See also \cite{GonzalezMirandaPeralta_TKFT}.
\end{remark}

\subsection{Cantor-based encodings}
The standard ternary Cantor set is the set of real numbers $x \in \mathbb{R}$ that 
can be written as
\[
x = \sum_{i=1}^{\infty} \epsilon_i 3^{-i}, \qquad \epsilon_i \in \{0,2\}.
\]
This Cantor set provides a way to encode sequences of natural numbers on the unit interval.

Following this idea, we introduce a convenient map that encodes tape states over the alphabet $\mathcal{A} = \{0,1\}$ as points of the closed interval $I = [0,1] \subset \mathbb{R}$. Let 
$t = \{t_n\}$ be a tape state with $t_n \in \mathcal{A}$ for all $n \in \Z$. We assign to $t$ the point in the Cantor set
\[
x_t =2\sum_{n=0}^\infty t_{n}3^{-(2n+1)}+2\sum_{n=1}^{\infty} t_{-n}3^{-2n}=2(t_03^{-1}+t_{-1}3^{-2}+t_13^{-3}+\cdots).
\]
Here, we have used the elements of $t$ as the ternary digits of our point in $I$, by alternating elements of the positive half of tape, $\{t_n\}_{n\geq 0}$, with those of the negative half, $\{t_n\}_{n<0}$. In other words, the right- and left-hand sides of the tape $t$ are encoded in the odd and even digits of the ternary expansion of $x_t$, respectively. Notice that since $t$ is finitely supported, the resulting number $x_t$ is rational, and hence constructible.

We would also like to add somehow the information about the position of the head into this encoding. In order to do so, we will embed a smaller copy of $I$ into itself for each $k\in \Z$ via an affine map $\tau_k:I\hookrightarrow I$. By having the length of $I_k:=\tau_k(I)$ decrease fast enough as $|k|$ goes to infinity, we can construct these embeddings so that the intervals $I_k$ are pairwise disjoint. For example, we can take the following maps $\tau_k:I\hookrightarrow I_k\subseteq I$:
    \[
        \tau_k(x) = 
        \begin{cases}
        3^{-(1-k)}(1+x), & k < 0, \\\\
        1+3^{-(1+k)}(-2+x), & k \geq 0.
        \end{cases}
    \]
This construction is illustrated in Figure~\ref{fig: embedding interval}. Using these intervals, we can encode the tape $t$ when the head state is at position $k$ as the constructible point $x_{t,k}:=\tau_k(x_t)\in[0,1]$.
\begin{figure}[h]
    \centering
    \includegraphics[width=0.8\linewidth]{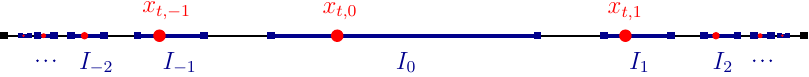}
    \caption{Encoding into the one-dimensional interval. Each interval represents a position of the head and the tape state is represented by a point in the interval. }
    \label{fig: embedding interval}
\end{figure}

\subsection{Computational operations as billiard dynamics}
In this section we show that the basic operations performed along each edge of the finite state machine, shifting the head and reading or writing a symbol, can be realized through billiard-like dynamics.  We will use these dynamics in the corridors that represent each transition of the Turing machine.

\begin{lemma}\label{lemma: shift}
    The shift operation, in which the head moves one cell left or right, can be realized through billiard-like dynamics. This means that there exists an explicit billiard wall whose induced dynamics are conjugate to the shift operation of a Turing machine.
\end{lemma}

\begin{proof}
    Using our encoding of the computation states into the one-dimensional interval, shifting the head becomes equivalent to mapping the points $x_{t,k}$ to $x_{t,k+\varepsilon}$ with $\varepsilon=-1,1$. We can assume without loss of generality that $\varepsilon=1$, and prove that we can find some explicit billiard walls that implement these dynamics.
    
    We will use the explicit expression for the embeddings $\tau_k:I\hookrightarrow I_k\subseteq I$:
    \[
        \tau_k(x) = 
        \begin{cases}
        3^{-(1-k)}(1+x), & k < 0, \\\\
        1+3^{-(1+k)}(-2+x), & k \geq 0.
        \end{cases}
    \]
    Notice that using these embeddings, each interval $I_k$ has length $3^{-(1+|k|)}$ and they are all disjoint. Moreover, for any tape $t$
    \[
    x_{t,k+1} =
        \begin{cases}
        3x_{t,k}, & \text{if } k < 0, \\\\
        x_{t,k}/3 + 2/3, & \text{if } k \ge 0.
        \end{cases}
    \]
    These transformations are affine and can be realized using billiard walls shaped like parabolas sharing a common focus, as suggested in \cite{Moore1} for the three-dimensional setting and illustrated in Figure~\ref{fig: shift}.
    \begin{figure}[h]
        \centering
        \includegraphics[width=0.75\linewidth]{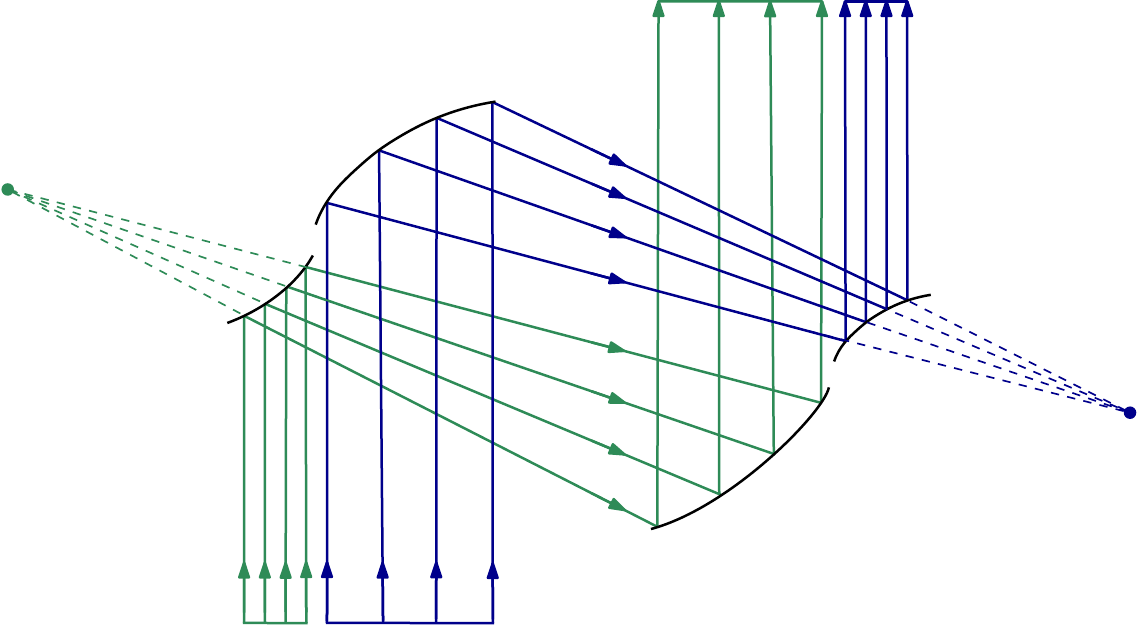}
        \caption{Billiard dynamics implementing a right shift. The green and blue rays correspond to the intervals $I_k$ with $k<0$ and $k\geq0$, respectively.}
        \label{fig: shift}
    \end{figure}
\end{proof}

    \begin{lemma}\label{lemma: read--write}
The read--write operation of a Turing machine can be implemented by a billiard. More precisely, there exists an explicit billiard table whose billiard dynamics are conjugate to the corresponding read--write operation.
\end{lemma}

\begin{proof}
    Let \(\delta\) be the transition function of a Turing machine. Let us consider a computation state encoded by a point \(x_{t,k}\in I\), where \(t=(t_n)_{n\in\mathbb{Z}}\) is the tape content and \(k\in\mathbb{Z}\) is the head position. Suppose that
    \[
        \delta(q,t_k)=(q',s,\varepsilon),
    \]
    so that the read--write operation consists in replacing the symbol \(t_k\) by \(s\), while keeping all other symbols unchanged. Then, at the level of the encoding, the read--write operation is equivalent to the mapping
    \[
        x_{t,k}\longmapsto x_{t',k},
    \]
    where \(t'_n=t_n\) for \(n\neq k\) and \(t'_k=s\).

    Notice that in other embeddings of Turing machines into dynamical systems, such as those in \cite{Moore1,Moore2}, where the tape is encoded into a two-dimensional Cantor set, the head is assumed to be always at position \(0\) of the tape. Therefore, the read--write operation involves only the leading ternary digit of the associated point. In our case, however, the symbol $t_k$ that we need to read and possibly modify is stored in a specific digit of the ternary expansion of $x_{t,k}$, namely, in the \((2k+1)\)-th or \((-2k)\)-th digit when \(k \geq 0\) or \(k < 0\), respectively. Thus, on the interval \(I_k\subset I\) associated with head position \(k\), the admissible computation states split into two alternating families of subintervals, each of them associated with the symbol \(0\) or \(1\)  at position \(k\). As \(|k|\) increases, the number of these intervals grows, while their lengths decrease.
    
    The billiard wall implementing the read--write operation is constructed as a curve whose segments over these subintervals are straight lines with negative or positive slope (depending on whether they correspond to reading a \(0\) or a \(1\), respectively). As a result, we can separate the states according to the symbol read and redirect them through different corridors of the billiard, as shown in Figure~\ref{fig: read--write}. Moreover, if it is necessary to change a \(0\) to a \(1\) (or vice versa), the angle of the straight segments of the wall can be slightly adjusted depending on \(k\), so that the resulting ray is displaced by the required factor.

    \begin{figure}[h]
      \centering
      \includegraphics[width=0.75\linewidth]    {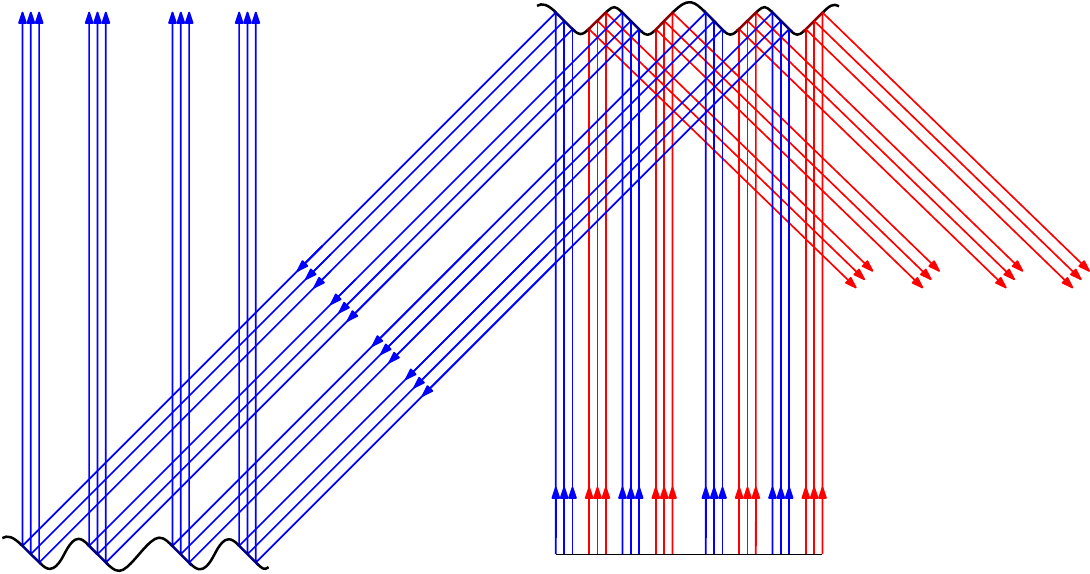}
      \caption{Billiard dynamics implementing the read--write operation. The blue and red rays correspond to reading a \(0\) or a \(1\), respectively.}
      \label{fig: read--write}
    \end{figure}

    Notice that, since there are finite gaps between the intervals where the computation states are, these straight segments can be joined smoothly, producing a differentiable curve. 

    For a more thorough explanation of how this curve is constructed, refer to the \nameref{appendix}.

\end{proof}

\begin{remark}
   Notice that the proofs of these two lemmas rely crucially on the specific encoding of the computational state into a one-dimensional interval and the structural features of billiard dynamics. In particular, if one adopted the alternative definition of a Turing machine in which the head is fixed at position zero and the tape is shifted instead, it would suffice to encode the tape \(t\) as a single point \(x_t \in I\). However, the interval map conjugate to the shift would then be extremely complicated and, to our understanding, impossible to realize within a physical system. The encoding we use avoids this obstruction at the cost of complicating the read--write map, as discussed in the proof of Lemma~\ref{lemma: read--write}. Despite this added complexity, the construction remains physically realizable by exploiting the power of billiard dynamics.
\end{remark}

\subsection{Proof of the main theorem}

\begin{proof}[Proof of Theorem \ref{thm: main result}]
    As discussed above, the strategy of the proof is to translate the directed graph encoding the Turing machine into a planar billiard table whose trajectories simulate its computation. Notice that we can assume, without loss of generality, that our Turing machine satisfies the condition discussed in Remark \ref{rem:normalizations}.

    First, as shown in Figure~\ref{fig: graph to billiard}, we represent each state $q_i$ by a segment $\Phi_i(I)\subset B$ obtained via a rational translation of the unit interval.  In the particular case of the initial and halting states, we choose the segments $\Phi_0(I)$ and $\Phi_{\mathrm{halt}}(I)$ to lie on the boundary of the table.  Thus the trajectory encoding the computation with input $t$ starts perpendicularly from the boundary at the computable point $\Phi_0(x_{t,0})$.  Similarly, if the computation halts with the head at position $k$ and output tape $t^*$, the trajectory meets the boundary orthogonally at the computable point $\Phi_{\mathrm{halt}}(x_{t^*,k})$.
    
    Using Lemmas~\ref{lemma: shift} and~\ref{lemma: read--write}, we can construct billiard corridors associated with each edge of the graph so that the billiard dynamics along that corridor implements precisely the read--write--shift operation prescribed by the transition. In this way, whenever a trajectory enters a corridor encoding an edge, the evolution of the tape and head position is updated exactly according to the Turing machine rule corresponding to that edge. 

    At each state $q_i$, several edges of the graph may enter the vertex. To reproduce this behavior in the billiard table, we introduce suitable billiard walls, of the same type as those used in the proof of Lemma~\ref{lemma: read--write}, that merge the incoming corridors into a single outgoing region, as illustrated in Figure~\ref{fig: graph to billiard}. The construction is arranged so that if the current computation state is $(t,q_i,k)$, then the billiard trajectory passes through the point $\Phi_i(x_{t,k})$.

    The reversibility of the Turing machine is essential at this stage of the construction. The billiard walls used to merge trajectories are designed by reversing the dynamics of the separating walls introduced in Figure~\ref{fig: read--write}: rather than splitting a single incoming trajectory into several outgoing ones, they identify several incoming corridors and funnel all corresponding trajectories into a single beam. This is only possible because of reversibility, which ensures that trajectories coming from different edges do not follow the same path.

    With this setup, the correspondence between computation and dynamics is exact. 
    Indeed, the billiard trajectory starting orthogonally at $\Phi_0(x_{t,0})$ reaches, again orthogonally to the boundary, the point $\Phi_{\mathrm{halt}}(x_{t^*,k})$ if the Turing machine with input tape $t$ halts with output tape $t^*$ and head position $k$. Otherwise, if the Turing machine does not halt, the billiard trajectory is not closed and keeps looping around the billiard table without ever entering the corridor leading to $\Phi_{\mathrm{halt}}(I)$. Moreover, the dynamics simulate the computation of the Turing machine step by step, in the sense that if
\[
(t_n,q_n,k_n)=\Delta_M^n(t_0,q_0,k_0),
\]
then the billiard trajectory starting orthogonally from
$\Phi_{q_0}(x_{t_0,k_0})$ passes successively through the points $\Phi_{q_n}(x_{t_n,k_n})$ for $n=0,1,2,\ldots$. In addition, it also follows from the construction that the trajectory only crosses the ``checkpoint'' intervals $\Phi_{q_i}(I)$ at these points, which corresponds to the steps of the associated computation.

Finally, to establish the global regularity of the billiard table, observe that the  Turing machine consists of finitely many states and transitions. Therefore, the 
billiard table requires only a finite number of corridors and walls to implement 
the shift and read--write dynamics. The walls implementing the shift operations 
are purely parabolic and hence smooth. For each transition that requires a read-write operation, the construction uses a single connected wall. As shown in the \nameref{appendix}, each of these read--write walls is smooth everywhere except at exactly two singular accumulation points. Because there are only finitely many read--write walls in the global construction, there are only finitely many singular points in total. Therefore, the global boundary of the billiard table is composed of a finite 
union of smooth arcs meeting at finitely many singular points.
\end{proof}

Applying this construction to a universal Turing machine, we get:
\begin{corollary}
    There exist two-dimensional billiard systems that are Turing complete in the sense of Definition \ref{def: Turing completeness}.
\end{corollary}

An interesting outcome of this construction is that the undecidability of the Halting problem \cite{turing} implies the undecidability of classical billiard systems. Directly from the main theorem, deciding whether a given trajectory enters a computable open set becomes undecidable. Moreover, from our particular construction of the billiard table, we also get the following result:
\begin{corollary}
There exists an explicitly constructible planar, bounded and finitely piecewise smooth billiard table for which deciding whether a trajectory starting at a computable point and direction is periodic is algorithmically undecidable.
\end{corollary}

\begin{proof}
Let us consider a billiard table associated with a universal Turing machine, constructed as in the proof of Theorem~\ref{thm: main result}. We claim that it is algorithmically undecidable to determine whether the billiard trajectory starting orthogonally at the \emph{computable} point $\Phi_0(x_{t,0})$ is periodic.

By construction, the billiard dynamics simulate the computation of the associated Turing machine with initial tape $t$. If the machine halts, then the trajectory reaches the boundary orthogonally at a point $\Phi_{\mathrm{halt}}(x_{t^*,k})$, where $t^*$ is the output tape. The orthogonal reflection reverses the motion, so the trajectory retraces its path and therefore forms a periodic orbit.

Suppose now that the machine does not halt. Then the corresponding billiard trajectory never reaches the halting region and continues moving through the corridors of the billiard table, reproducing the successive steps of the computation. Since the Turing machine is reversible, the induced dynamics on computation states are injective. Therefore, if at some stage the computation were to enter a periodic loop, then injectivity would force all previous states in the orbit to be uniquely determined by the states in that loop. In particular, the initial configuration would also belong to the same periodic orbit, and hence the computation would have to return to its initial state \(q_0\). But this is impossible, because no transition enters \(q_0\). Therefore, the computation can never become periodic, and so the associated billiard trajectory never closes up.

Since the halting problem for a universal Turing machine is undecidable, it follows that deciding whether this billiard trajectory is periodic is also undecidable.
\end{proof}

From our proof method, it also follows immediately that the system follows the alternative definition of Turing completeness presented in \cite{tao1,CardonaMirandaPeraltaPresas,CardonaMirandaPeralta_euclidean}. Namely, it is possible to algorithmically associate the halting problem of any Turing machine to the reachability problem for certain billiard trajectories  of a planar billiard table. More formally, we have the following result:
\begin{corollary}
    There exists an explicitly constructible planar, bounded and finitely piecewise smooth billiard table $B$ where the reachability problem for certain starting points is undecidable in the following sense.  Given any Turing machine $M$, any integer $N\ge 0$, any input tape $t$, any integer $k\in\mathbb{Z}$ and any finite string $(t_{k-N}^*,\ldots,t_{k+N}^*)$ of symbols of the alphabet, we can algorithmically construct a point $p\in\partial B$ and an open set $U\subset\mathbb{R}^2$ such that the billiard trajectory starting perpendicularly from $p$ intersects $U$ if and only if $M$ halts with its head at position $k$ and outputs a tape whose entries $k-N,\ldots,k+N$ are $t_{k-N}^*,\ldots,t_{k+N}^*$.
\end{corollary}
\begin{proof}
    It suffices to prove this result for a fixed universal Turing machine $M$. Applying Theorem \ref{thm: main result}, we can construct the desired billiard table whose dynamics simulate the computations of $M$. Moreover, since $M$ is a universal Turing machine, with this fixed billiard table $B$ one can simulate any other Turing machine. With this setup, the correspondence between computation and dynamics is exact. Indeed, using the construction given in the proof of Theorem \ref{thm: main result}, for an input tape $t$ and an integer $k\in\mathbb{Z}$, the billiard trajectory starting orthogonally at $p:=\Phi_0(x_{t,0})$ reaches, again orthogonally to the boundary, the point $\Phi_{\mathrm{halt}}(x_{t^*,k})$ if the Turing machine with input tape $t$ halts with output tape $t^*$ and head position $k$. Otherwise, if the Turing machine does not halt, the billiard trajectory is not closed and keeps looping around the billiard table without ever entering the corridor leading to $\Phi_{\mathrm{halt}}(I)$.

    This means that, given \(N \geq 0\) and symbols \(t^*_{k-N},\dots,t^*_{k+N}\), we consider all points of the form \(\Phi_{\mathrm{halt}}(x_{t,k})\) such that the tape \(t\) satisfies
    \[
        t_{k-N}=t^*_{k-N},\dots,t_{k+N}=t^*_{k+N}.
    \]
    By construction of the Cantor-type encoding, these points form a finite union of disjoint subintervals of \(\Phi_{\mathrm{halt}}(I)\), corresponding exactly to the possible output tapes with the prescribed symbols in positions \(k-N,\dots,k+N\) and the head position $k$. We then choose \(U\) to be a sufficiently small computable open neighborhood of this finite union so that it contains no other halting states with different symbols in the prescribed positions. With this choice, the billiard trajectory enters \(U\) if and only if the Turing machine halts with head position \(k\) and output tape containing the symbols \(t^*_{k-N},\dots,t^*_{k+N}\) at positions \(k-N,\dots,k+N\).
\end{proof}

\section{Physical Applications}

Billiard dynamics arises naturally in classical mechanics, both as an exact reformulation of particle systems with hard constraints and as a singular limit of smooth Hamiltonian flows. Thus, the computational universality of planar billiards is not only a feature of idealized reflection models: it points to intrinsic algorithmic limitations in the long-time prediction of physically meaningful Hamiltonian systems admitting billiard-type limits.

A basic example is provided by hard-sphere systems. A gas of \(N\) hard spheres of radius \(r\) moving in a bounded domain \(D\subset \mathbb R^d\) can be described exactly as a billiard flow in configuration space. If \(q_i\in D\) denotes the position of the \(i\)-th particle, the admissible configurations form
\[
\mathcal C =
\left\{
(q_1,\ldots,q_N)\in D^N \;\middle|\; |q_i-q_j|\geq 2r
\ \text{for all } i\neq j
\right\}.
\]
The motion is free in the interior of \(\mathcal C\), while collisions correspond to elastic reflections at \(\partial \mathcal C\) \cite{ChernovMarkarian2006}. Consequently, billiard systems appear directly in kinetic theory, and invariant billiard subsystems may already exhibit algorithmically undecidable qualitative behavior.

Billiard-type dynamics also appears in celestial mechanics through singular limits associated with close encounters. In the Newtonian \(N\)-body problem,
\[
H(q,p)
=
\sum_{i=1}^N \frac{|p_i|^2}{2m_i}
-
\sum_{1\leq i<j\leq N}
\frac{m_i m_j}{|q_i-q_j|},
\]
binary collisions occur along the sets \(q_i=q_j\). Regimes dominated by repeated near-collisions can be modeled by collision chains: trajectories composed of Keplerian arcs joined by reflection-type matching conditions at collision manifolds. This idea goes back to Poincar\'e's solutions of the second species and was developed rigorously by Bolotin and collaborators \cite{Poincare1899,BolotinNegrini2013,Bolotin2017}. In such limits, the dynamics naturally carries symbolic itineraries recording successive near-collisions, in close analogy with billiard dynamics.

A concrete example is provided by the planar circular restricted three-body problem. After reduction to a rotating frame and rescaling near a close encounter with one primary, the Hamiltonian takes the schematic form
\[
H(q,p)
=
\frac{1}{2}\lvert p-\omega\times q\rvert^2
-
\frac{1}{|q|}
+
\mu V(q)
+
O(\mu),
\]
where the Kepler term dominates near collision and the perturbation $\mu V(q)$ encodes the influence of the distant primary. In the corresponding singular limit, near-collision trajectories satisfy reflection-type matching conditions at the collision set. Symbolic itineraries analogous to those arising in billiard systems then organize the global dynamics.

Finally, billiards arise as singular limits of smooth Hamiltonian systems with steep confining potentials,
\[
H_\varepsilon(q,p)
=
\frac12 |p|^2 + V_\varepsilon(q),
\qquad
V_\varepsilon(q)\to +\infty
\quad\text{as } q\to \partial B.
\]
As \(\varepsilon\to0\), the corresponding Hamiltonian flow converges, away from grazing trajectories, to the billiard flow in \(B\) on finite time intervals. This soft-wall approximation embeds billiards within smooth Hamiltonian mechanics and supports the interpretation of billiard undecidability as a phenomenon that may arise as a limit of physically natural systems, in the spirit of Moore's perspective on undecidability in dynamical systems \cite{Moore1}.

\section{Conclusions}
Billiard systems have long served as fundamental models in dynamical systems, valued for their simplicity and rich behavior. In this work, we show that their expressive power extends beyond classical notions of integrability and chaos \cite{Birkhoff1927,Sinai1963}: planar billiards can simulate arbitrary computation. As a result, basic dynamical questions, including the existence of periodic trajectories and reachability, become algorithmically undecidable \cite{turing,Moore1}. This viewpoint suggests that computational complexity is intrinsic to billiard dynamics rather than a mathematical artifact. Billiard-type motion arises naturally across physics, often as an effective description in singular or near-collision regimes, notably in celestial mechanics and kinetic theory. In such settings, undecidability imposes fundamental limits on long-term prediction, complementing and extending the traditional role of chaos.

Billiards, therefore, provide a transparent arena for exploring the interface between geometry, dynamics, and computation. An open and provocative question is whether these classical algorithmic barriers survive quantization. If traces of classical undecidability persist in quantum billiards, they would point to limits of predictability that transcend the classical–quantum divide.

\bibliographystyle{plain}
\bibliography{bib}

\section*{Appendix}\label{appendix}
The purpose of this appendix is to provide the missing details in the proof of Lemma \ref{lemma: read--write} from the main text, thus completing the argument that the read--write operation can be realized by means of billiard dynamics. Namely, we give an explicit expression for the straight-segments section of the read--write wall. Then, we prove that the read--write dynamics are well-defined in the sense that after bouncing off the wall, any trajectory encoding a computational state does not collide again with any other part of the wall. Finally, we show that the wall can be made smooth by joining the straight segments smoothly in the gaps between them, except at the two accumulation points at the endpoints of the interval, where the wall can only be extended continuously.

We start by giving an explicit expression for the straight segments, the part of the read--write wall where the trajectories associated with computational states collide. Recall first that, according to our definition of a Turing machine, applying the read--write operation to a computation state $(t,k)$ requires inspecting and possibly modifying the tape symbol $t_k$. Under our Cantor set--type encoding, this amounts to detecting and, if necessary, changing a specific digit in the ternary expansion associated with the tape before applying the embedding $\tau_k$. For this reason, we argued that, inside the interval $I_k$ corresponding to the head position $k$, the admissible computation states decompose into two alternating families of subintervals, depending on whether the symbol at position $k$ is $0$ or $1$. We now make this decomposition explicit.

Fix $k\in\mathbb{Z}$ and, for simplicity, assume $k\geq 0$ (the case $k<0$ is completely analogous). We begin by considering the encoding of the tape $t$ alone as a point $x_t\in I$, prior to applying the embedding $\tau_k$. Our goal is to describe two families of Cantor set blocks corresponding to the conditions $t_k=0$ and $t_k=1$. Recall that, in our encoding, the symbol $t_k$ is stored in the $(2k+1)$-th digit of the ternary expansion of $x_t$. To simplify notation, we introduce
\[
x_{\epsilon_1\ldots\epsilon_l}
:=2\sum_{n=1}^{l}\epsilon_n\,3^{-n},
\]
where each $\epsilon_n\in\{0,1\}$.

With this notation, the Cantor set blocks corresponding to tapes with $t_k=0$ are given by
\[
I_{\epsilon_1\ldots\epsilon_{2k}0}
=\bigl[x_{\epsilon_1\ldots\epsilon_{2k}0},
\,x_{\epsilon_1\ldots\epsilon_{2k}0}+3^{-(2k+1)}\bigr],
\qquad \epsilon_n\in\{0,1\},\ n=1,\ldots,2k,
\]
while those corresponding to tapes with $t_k=1$ are
\[
I_{\epsilon_1\ldots\epsilon_{2k}1}
=\bigl[x_{\epsilon_1\ldots\epsilon_{2k}1},
\,x_{\epsilon_1\ldots\epsilon_{2k}1}+3^{-(2k+1)}\bigr],
\qquad \epsilon_n\in\{0,1\},\ n=1,\ldots,2k.
\]

We now apply the embedding $\tau_k$, using its explicit expression from the proof of Lemma~1. Define
\[
x_{\epsilon_1\ldots\epsilon_l}^k
:=\tau_k(x_{\epsilon_1\ldots\epsilon_l})
=1+3^{-(1+k)}\bigl(-2+x_{\epsilon_1\ldots\epsilon_l}\bigr).
\]
Then the Cantor set blocks inside $I_k$ corresponding to the condition $t_k=s$, with $s\in\{0,1\}$, are
\[
I_{\epsilon_1\ldots\epsilon_{2k}s}^k
:=\tau_k\bigl(I_{\epsilon_1\ldots\epsilon_{2k}s}\bigr)
=\bigl[x_{\epsilon_1\ldots\epsilon_{2k}s}^k,
\,x_{\epsilon_1\ldots\epsilon_{2k}s}^k+3^{-(3k+2)}\bigr],
\]
again with $\epsilon_n\in\{0,1\}$ for $n=1,\ldots,2k$. There are $2^{2k}$ such intervals, each of length $3^{-(3k+2)}$.

We now describe the billiard walls implementing the read--write operation. We assume that incoming billiard trajectories move upward, with positive velocity in the $y$-direction, and have horizontal coordinate $x\in[0,1]$, where the computation state $(t,k)$ is encoded as $x_{t,k}$. We first treat the case in which no symbol change occurs, so that the goal is simply to distinguish between $t_k=0$ and $t_k=1$.

After interacting with the billiard walls, we want trajectories corresponding to $t_k=0$ to continue vertically with horizontal coordinate $x_{t,k}-2$, while those corresponding to $t_k=1$ should continue vertically with horizontal coordinate $x_{t,k}+2$. To achieve this separation, trajectories first collide with straight walls placed above the intervals $I_{\epsilon_1\ldots\epsilon_{2k}s}^k$, with slope $-1$ for $s=0$ and slope $+1$ for $s=1$.

For $t_k=0$, we define the walls
\[
W_{\epsilon_1\ldots\epsilon_{2k}0}^k(x)
=2+x_{\epsilon_1\ldots\epsilon_{2k}0}^k
+\frac{3^{-(3k+2)}}{2}-x,
\qquad x\in I_{\epsilon_1\ldots\epsilon_{2k}0}^k.
\]
Each such wall is a line segment of slope $-1$, centered horizontally above the corresponding interval $I_{\epsilon_1\ldots\epsilon_{2k}0}^k$ and located at height $y=2$.

Similarly, for $t_k=1$, we take walls of slope $+1$ given by
\[
W_{\epsilon_1\ldots\epsilon_{2k}1}^k(x)
=2-x_{\epsilon_1\ldots\epsilon_{2k}1}^k
-\frac{3^{-(3k+2)}}{2}+x,
\qquad x\in I_{\epsilon_1\ldots\epsilon_{2k}1}^k.
\]
These walls deflect trajectories corresponding to $t_k=0$ to the left and those corresponding to $t_k=1$ to the right.

Once the trajectories have been separated, we must redirect them so that they again travel vertically. For instance, in the case $t_k=0$, this is achieved by introducing a second family of walls
\[
\widetilde{W}_{\epsilon_1\ldots\epsilon_{2k}0}^k(x)
=x_{\epsilon_1\ldots\epsilon_{2k}0}^k
+\frac{3^{-(3k+2)}}{2}-(x+2),
\]
defined for
\[
x\in\bigl[x_{\epsilon_1\ldots\epsilon_{2k}0}^k-2,
\,x_{\epsilon_1\ldots\epsilon_{2k}0}^k-2+3^{-(3k+2)}\bigr].
\]
A trajectory starting at $x_{t,k}$ with $t_k=0$ first reflects off the wall $W_{\epsilon_1\ldots\epsilon_{2k}0}^k$, then travels along a line of slope $-1$ until it reaches $\widetilde{W}_{\epsilon_1\ldots\epsilon_{2k}0}^k$. Since the two walls are parallel, the outgoing trajectory is vertical, with horizontal coordinate $x_{t,k}-2$, as desired (see Figure \ref{fig: read--write no change}). A similar construction can be used to have the outgoing trajectory corresponding to $t_k=1$ travel vertically with horizontal coordinate $x_{t,k}+2$.

\begin{figure}[h]
    \centering
    \includegraphics[width=0.8\linewidth]{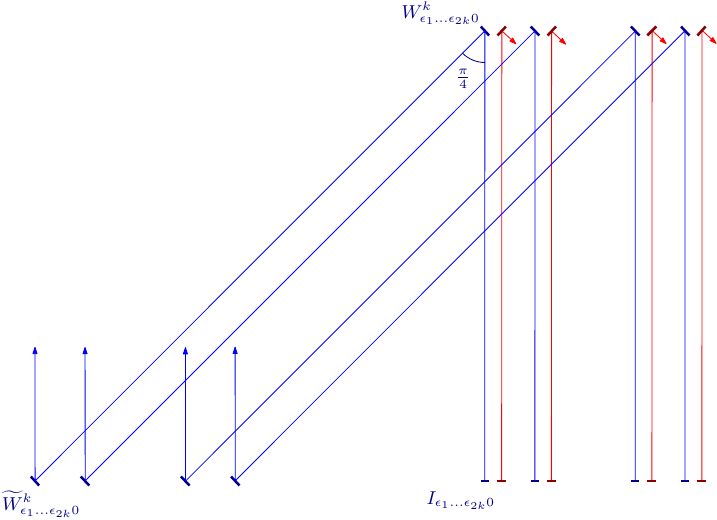}
    \caption{Billiard dynamics implementing the read–write operation with no symbol change. The blue and
red rays correspond to reading a 0 or a 1, respectively.}
    \label{fig: read--write no change}
\end{figure}

We now turn to the case in which a symbol change is required. For simplicity, suppose that we want to modify the tape symbol at position $k$ from $t_k=0$ to $t'_k=1$. In terms of our encoding, this means that an incoming vertical trajectory with horizontal coordinate $x_{t,k}$ should be transformed, after interacting with the read--write construction, into a vertical trajectory with horizontal coordinate
\[
x_{t,k}+2\cdot 3^{-(3k+2)}-2.
\]
Compared with the pure read case, this requires a slight horizontal displacement of the trajectory. Therefore, the separating wall can no longer have slope exactly $-1$. Instead, we choose the angle $\alpha_k$ with the horizontal so that
\[
\tan\alpha_k=\frac{1}{1+ 3^{-(3k+2)}}.
\]
This deviation from slope $-1$ is very small, especially for large values of $k$, and in the limit $k\to\infty$ we recover the original configuration. We can observe this in Figure \ref{fig: read--write  change}, where the small displacement is displayed compared to the original no-rewrite wall.

\begin{figure}[h]
    \centering
    \includegraphics[width=0.8\linewidth]{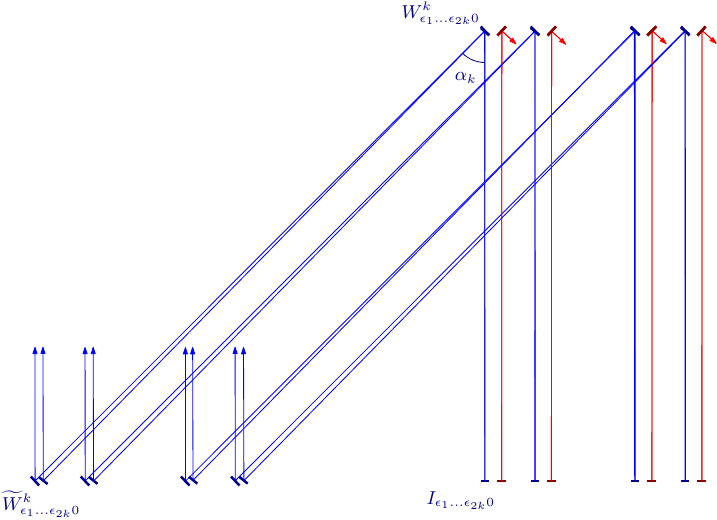}
    \caption{Billiard dynamics implementing the read–write operation with symbol change compared to the no-change situation.}
    \label{fig: read--write  change}
\end{figure}

We can now prove that the read--write operation can be performed with these walls, and in particular, that the corresponding trajectories do not collide with any other wall from this construction.

\begin{lemma}
    Let us consider any $k\in \Z$ (we can assume $k\geq 0$), $t\in \Lambda$ and $s\in \{0,1\}$. Assume that $x_{t,k}=x^k_{\epsilon_1\ldots \epsilon_2{k}t_k}$ for some $\epsilon_1\ldots \epsilon_{2k}\in \{0,1\}$. We claim that the billiard trajectory starting at $(x^k_{\epsilon_1\ldots \epsilon_{2k}t_k},0)$ with vertical direction, travels from the wall
$W_{\epsilon_1\ldots\epsilon_{2k}t_k}^k$ to the corresponding
$\widetilde{W}_{\epsilon_1\ldots\epsilon_{2k}s}^k$ without colliding with any other wall of the construction.
\end{lemma}
\begin{proof}
    First, since all walls lie at the same height, this reduces to checking that, after bouncing off $W_{\epsilon_1\ldots\epsilon_{2k}s}^k$, the trajectory descends sufficiently fast so as not to intersect any other wall $W_{\epsilon'_1\ldots\epsilon'_{2k'}s'}^{k'}$.

    Observe that trajectories corresponding to a symbol change have a steeper downward slope than those in the no-rewrite case. As a result, avoiding collisions is strictly easier in the rewriting case. It therefore suffices to carry out the check for the read--only configuration, where the walls have slopes $\pm 1$.

    Consider a trajectory that bounces off a wall $W_{\epsilon_1\ldots\epsilon_{2k}0}^k$, i.e., the case $t_k=0$. Until it reaches $\widetilde{W}_{\epsilon_1\ldots\epsilon_{2k}0}^k$, this trajectory remains parallel to all walls of the form $W_{\epsilon'_1\ldots\epsilon'_{2k'}1}^{k'}$ (or their slope is larger when the symbol $1$ is modified by the transition). Since such walls are separated by a horizontal distance of at least $2\cdot 3^{-(3k+2)}$, there is no risk of collision with them. We therefore only need to check that the trajectory does not hit any wall
$W_{\epsilon'_1\ldots\epsilon'_{2k'}0}^{k'}$ lying to its left.

    It is enough to consider the extremal situation. Namely, we check that the trajectory bouncing off the left endpoint of
$W_{\epsilon_1\ldots\epsilon_{2k}0}^k$, located at
\[
\bigl(x_{\epsilon_1\ldots\epsilon_{2k}0}^k,\,
2+3^{-(3k+2)}/2\bigr),
\]
does not intersect the right endpoint of any wall
$W_{\epsilon'_1\ldots\epsilon'_{2k'}0}^{k'}$, located at
\[
\bigl(x_{\epsilon'_1\ldots\epsilon'_{2k'}0}^{k'}+3^{-(3k'+2)},\,
2-3^{-(3k'+2)}/2\bigr).
\]
This amounts to verifying the inequality
\[
x_{\epsilon_1\ldots\epsilon_{2k}0}^k-
\bigl(x_{\epsilon'_1\ldots\epsilon'_{2k'}0}^{k'}+3^{-(3k'+2)}\bigr)
>
\frac{3^{-(3k+2)}}{2}+\frac{3^{-(3k'+2)}}{2}.
\]

We first consider the case $k=k'$. By construction, if
$x_{\epsilon_1\ldots\epsilon_{2k}0}^k\neq
x_{\epsilon'_1\ldots\epsilon'_{2k}0}^{k}$, then
\[
x_{\epsilon_1\ldots\epsilon_{2k}0}^k-
x_{\epsilon'_1\ldots\epsilon'_{2k}0}^{k}
\geq 2\cdot 3^{-(3k+1)}.
\]
It follows that
\[
x_{\epsilon_1\ldots\epsilon_{2k}0}^k-
\bigl(x_{\epsilon'_1\ldots\epsilon'_{2k}0}^{k}+3^{-(3k+2)}\bigr)
\geq 3^{-(3k+2)}(6-1)
>3^{-(3k+2)},
\]
which is more than enough to guarantee the desired separation.

We now turn to the case $k\neq k'$. In this situation,
$x_{\epsilon_1\ldots\epsilon_{2k}0}^k\in I_k$ and
$x_{\epsilon'_1\ldots\epsilon'_{2k'}0}^{k'}\in I_{k'}$.
Since the distance between $I_{k'}$ and $I_{k'+1}$ is
$3^{-(k'+2)}$, the distance between $I_{k'}$ and $I_k$ is at least this large. Consequently,
\[
x_{\epsilon_1\ldots\epsilon_{2k}0}^k-
\bigl(x_{\epsilon'_1\ldots\epsilon'_{2k'}0}^{k'}+3^{-(3k'+2)}\bigr)
\geq 3^{-(k'+2)}
\geq 3^{-(3k'+2)}
>
\frac{3^{-(3k+2)}}{2}+\frac{3^{-(3k'+2)}}{2},
\]
where in the last inequality we used the fact that $k>k'$.

We conclude that, after bouncing off $W_{\epsilon_1\ldots\epsilon_{2k}0}^k$, the trajectory does not intersect any other wall
$W_{\epsilon'_1\ldots\epsilon'_{2k'}s'}^{k'}$. By symmetry, the same argument applies to walls of the form
$W_{\epsilon_1\ldots\epsilon_{2k}1}^k$ and to possible collisions with
$\widetilde{W}_{\epsilon_1\ldots\epsilon_{2k}s}^k$, as well as to the case $k<0$.
\end{proof}

We now explain how to join the straight segments so as to obtain a wall with the desired regularity that implements the read--write operation.
\begin{proposition}
   There exists a connected wall implementing the read--write operation through billiard
dynamics. This wall is smooth away from its endpoints, where it extends only
continuously.
\end{proposition}
\begin{proof}

First, we note that the estimates used above to rule out unwanted collisions are far from sharp, meaning there is a positive margin of safety. This allows us to create a smooth wall by joining the straight segments with smooth arcs inside the gaps between them. Since no 
computational state is encoded in these gaps, the wall can be modified  there without affecting the simulated computation, provided the smooth interpolations remain sufficiently close to the linear segments to preserve 
the non-collision estimates. This smoothing can be performed cleanly everywhere except at the two limit points corresponding to $k \to \pm\infty$.  Indeed, our construction is locally finite: any compact subinterval away from 
the endpoints intersects only finitely many intervals $I_k$, and since each 
$I_k$ contains a finite number of straight segments, any such subinterval contains 
only finitely many oscillations. Thus, we can remove every interior singularity, 
rendering the wall $C^\infty$-smooth everywhere except at the two endpoint accumulation points, where any neighborhood necessarily contains infinitely many oscillations.

 It remains to analyze the behavior of the wall near these two endpoints. Although the smoothing procedure cannot be extended there, the wall still admits a well-defined continuous extension. Indeed, let us consider the portion of the wall corresponding to a fixed interval $I_k$ and consider the case $k\geq 0$, i.e., the limit point on the right corresponding to $k\to \infty$ (the case $k\to -\infty$ is analogous). The oscillatory linear segments of this portion are confined to the Cantor blocks of order $2k+1$. Since there are $2^{2k}$ such blocks, each of length $3^{-(2k+1)}$, their total length is
\[
2^{2k}3^{-(2k+1)}=\frac13\left(\frac49\right)^k\xrightarrow[k\to \infty]{}0.
\]
Moreover, each zig--zag segment has vertical size on the order of $3^{-(3k+2)}$, which also tends to zero as $|k|\to\infty$. Therefore, both the horizontal support and the vertical amplitude of the oscillations vanish near the endpoints of the interval. It follows that the restriction of the wall to $I_k$ converges uniformly to the straight segment joining the endpoints of $I_k$. Consequently, the entire billiard wall extends continuously to the two accumulation points, although these remain the only points where the boundary is not smooth. Indeed, as one approaches either endpoint, the wall still contains infinitely many linear portions with slopes alternating between $1$ and $-1$ (or slopes asymptotically approaching these values, if there is a symbol change). Consequently, the tangent direction oscillates indefinitely and does not admit a limit at either accumulation point. Therefore, the present construction cannot be extended to a $C^1$ curve there.

 \begin{figure}[h]
        \centering
        \includegraphics[width=0.7\linewidth]{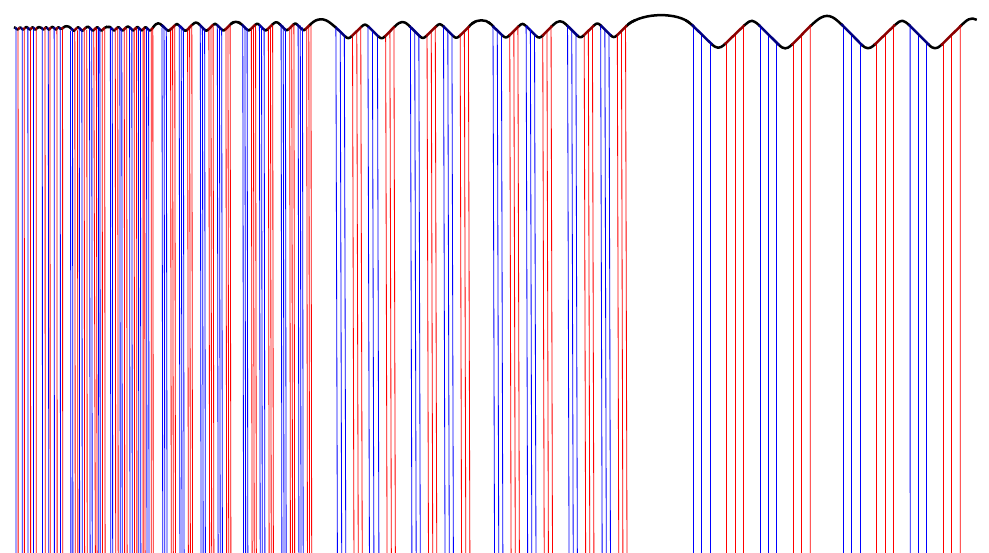}
        \caption{How the billiard wall implementing the read--write operation changes as $k\to -\infty$. }
        \label{fig: read--write 2}
    \end{figure}

 \end{proof}
\end{document}